\documentclass[12pt,reqno]{amsart}
\usepackage{amsmath}
\usepackage{amsfonts}
\usepackage{amssymb, amsthm, epsfig}
\usepackage{mathrsfs}

\usepackage[utf8]{inputenc}
\usepackage[T1]{fontenc}
\usepackage{geometry}
\usepackage{graphicx}
\usepackage{verbatim}
\usepackage{color}

\usepackage[colorlinks, linkcolor=blue, citecolor=blue]{hyperref}
\usepackage[notref,notcite,color,draft]{showkeys}
\theoremstyle{plain}
\newtheorem{thm}{Theorem}[section]

\newtheorem{lem}[thm]{Lemma}

\theoremstyle{definition}

\theoremstyle{remark}
\newtheorem{rem}{Remark}[section]

\numberwithin{equation}{section}

\allowdisplaybreaks

\newcommand{\norm}[1]{\left\Vert#1\right\Vert}
\newcommand{\abs}[1]{\left\vert#1\right\vert}
\newcommand{\set}[1]{\left\{#1\right\}}
\newcommand{\qnt}[1]{\left(#1\right)}

\newcommand{\bracket}[1]{\left[#1\right]}

\newcommand{\dif}{\mathrm{d}}

\DeclareSymbolFont{lettersA}{U}{pxmia}{m}{it}
\DeclareMathSymbol{\piup}{\mathord}{lettersA}{"19}

\makeatletter

\newcommand{\Rmnum}[1]{\expandafter\@slowromancap\romannumeral#1@}
\makeatother
\begin{document}
\title[Mixed boundary problem in a thin domain]
{Mixed type boundary value problem of elliptic equation in a thin domain}

\author{Dian Hu}
\address{D. Hu: School of Sciences, East China University Of Science And Technology,
Shanghai, 200237, China.}
\email{\tt hudianaug@qq.com and hudianaug@gmail.com}

\author{Genggeng Huang}
\address{G. Huang: School of Mathematical Sciences, Fudan University, Shanghai, 200433, China.}
\email{\tt genggenghuang@fudan.edu.cn}

\keywords{}
\subjclass[2000]{35L65,35L67,35M10,35B35,76H05,76N10}
\date{\today}

\begin{abstract}
In this paper, we prove the a priori estimates for two-dimensional second order homogeneous linear elliptic equations in a narrow region. In a crescent-shaped area, part of the boundary is subject to an oblique derivative boundary condition, while the other part of the boundary is subject to a Dirichlet boundary condition. We show that, as the crescent-shaped area collapses into a segment under suitable conditions, the boundary value problem obeys uniform Schauder estimates and induces an asymptotic estimate.
\end{abstract}

\maketitle

\section{Introduction}\label{sect:Introduction}
In this paper, we investigate second-order elliptic equations in a 2D crescent-shaped region with mixed boundary value condition. Consider a function $y=f(x)$ satisfying 
\begin{equation}
f(0)=f(1)=0,\quad 
f(x)>0,\quad x\in (0, 1).
\end{equation}
A crescent-shaped domain $\Omega$ is defined as
\begin{equation}\label{eq:Omega}
\Omega:=\set{(x, y)| x\in(0, 1), 0< y<f(x)},
\end{equation}
with the boundaries
\begin{equation}\label{eq:OmegaB}
\partial_+\Omega:=\set{(x, f(x))|x\in[0, 1]},\quad \partial_0\Omega:=\set{(x, 0)|x\in[0, 1]}.
\end{equation}

We study the following mixed boundary value problem of linear elliptic equation
\begin{equation}\label{eq:PotentialEquations}
\begin{split}
&\mathcal L u:=A(x, y)u_{xx}+B(x, y)u_{xy}+C(x, y)u_{yy}+D(x, y)u_x+E(x, y)u_y\\
&\quad\quad=0,\quad \text{in}\quad \Omega,\\
& u(x, y)=\varphi(x),\quad \text{on}\quad \partial_+\Omega,\\
& u_y(x, y)+G(x)u_x(x, y)=0,\quad \text{on}\quad \partial_0\Omega.
\end{split}
\end{equation}
The coefficients $A, B,C,D,E,G$ satisfy following conditions
\begin{equation}\label{eq:uniformelliptic}
\begin{split}
&A\xi_1^2+B\xi_1\xi_2+C \xi_2^2\ge \lambda |\xi|^2,\quad \mbox{for}~ \xi\in \mathbb R^2\\
&\|A\|_{C^{\gamma}([0,1]^2)}+\|B\|_{C^{\gamma}([0,1]^2)}+\|C\|_{C^{\gamma}([0,1]^2)}\\
+ &\|D\|_{C^{\gamma}([0,1]^2)}+\|E\|_{C^{\gamma}([0,1]^2)}+\|G\|_{C^{2,\gamma}([-1, 2])}\le \Lambda
\end{split}
\end{equation}
for some positive constants $\lambda,\Lambda$ and $\gamma\in (0,1)$.

Assume that $f$ satisfies 
\begin{equation}\label{eq:fEstimates}
\|f\|_{C^{2, \gamma}([0, 1])}\le \sigma,
\end{equation}
and
\begin{equation}\label{eq:FiniteGrowCondition}
\begin{cases}
\inf\limits_{x\in[0, 1]}\abs{\frac{f(x)}{\sin(\pi x)}}=\Pi>0,\\
\norm{f}_{C^2([0, 1])}=C_f\Pi=\bar{\Pi}.
\end{cases}
\end{equation}
Then we have
\begin{thm}\label{thm:GlobalEstimates}
Let $u$ be the solution of \eqref{eq:PotentialEquations}. Then there exists $\sigma_0>0$ such that for any $\sigma\in (0,\sigma_0]$, one has
\begin{equation}\label{eq:ObliqueDerivativeSchauder}
\norm{u}_{C^{2, \gamma}(\overline \Omega)}\le \bar{C}\norm{\varphi}_{C^{2, \gamma}([0,1])},
\end{equation}
where the constant $\bar{C}$ depends only on $\gamma,\lambda,\Lambda$ and $C_f$.
\end{thm}

The Dirichlet and oblique derivative boundary value problems have been systematically studied in a vast literature. We refer readers to, for example \cite{GilbargTrudinger,HanLin,LadyzenskajaSolonnikovUral'tseva} and references
therein. Furthermore, mixed type boundary value problems in nonsmooth regions have also been discussed, see \cite{AzzamKreyszig,Grisvard,KozlovMaz'yaRossman,Lieberman1,Miller}. 
%The existence, uniqueness, and a priori estimates of these boundary value problems play the role of fundamental tools in solving further complicated problems. 
Roughly speaking, the Schauder theory for boundary value problem of second-order elliptic equations says that if all the coefficients and data are Hölder continuous, then the second order Hölder norms of the solution can be estimated by the inhomogeneous term and boundary values. How the constants of estimates depend on the settings of the region, the coefficients and the boundary value is a complicated problem.
In this paper, by Theorem \ref{thm:GlobalEstimates}, we confirm that the estimate remains uniform even when the region for the mixed boundary value problem degenerates into a line segment in a specific way. Such an a priori estimate ensures uniform equicontinuity of the derivatives of the solution as the region shrinks. In other words, the rate of change of the solution and its derivatives is bounded. As the region degenerates into a line segment, the variation of the solution and its derivatives along specific directions diminishes. Thus, the solution converges to an asymptotic state. For our mixed boundary value problem, the compatibility conditions are satisfied at the corners. Otherwise, uniform estimates cannot be expected. For example, generically, if Dirichlet conditions are given along $\partial_0\Omega$ and $\partial_+\Omega$, then normal derivative would blow up as the region shrinks. 

To our knowledge, for elliptic boundary value problem in a thin domain, its eigenvalue has been discussed by variational in many works, see \cite{ArrietaNakasatoPereira,BorisovFreitas,FerraressoProvenzano,FriedlanderSolomyak,Grieser,JimboKurata,Krejcirik,KuchmentZeng,PereiraSilva,Post,RubinsteinSchatzman1,RubinsteinSchatzman2} and references therein. The motivation for our paper comes from the research of free boundary problems for elliptic conservation law systems. This type free boundary problem of elliptic partial differential equation does not have a variational structure, so we cannot establish the existence of solutions through variational methods, see \cite{AltCaffarelli,Caffarelli,CaffarelliSalsa,Lin}. We need a solution that complies with the structure of conservation laws, to serve as the starting point of a continuity method, to obtain a general solution, see \cite{BaeXiang,ChenGQ,ChenFeldman1,ChenFeldman2,ChenFeldmanXiang,ChenS,EllingLiu}. As the supersonic incoming flow accelerates to the limit speed, the free boundary area after the shock might shrink to a low-dimensional surface, see \cite{Anderson,ChenKuangXiang,ChenXinYin,HuZhang,JinQuYuan1,KuangXiangZhang,LiWittYin,QuYuanZhao,Van,WangZhang,XinYin,XuYin}. In this process, the discussion for a priori estimates and asymptotic behavior of the solution would help us to investigate the flow field and the asymptotic states might play the role of initial solutions for the continuity method.

The organization of this paper is as follows. In Section 2, under the assumptions of local straight boundary condition of $\partial_+\Omega$ and normal derivative boundary value condition on $\partial_0\Omega$, we proved our local estimate for \eqref{eq:ObliqueDerivativeSchauder}. In Section 3, by applying three coordinate transformations, we extend the estimate from Section 2 to the case of general boundary conditions and proved our main theorem. In Section 4, we present a similar conclusion for a weighted norm. In Section 5, we prove an asymptotic estimate as $\|f\|_{C^{2, \gamma}([0, 1])}\rightarrow0$.

\section{Local Straight Corner}\label{sect:LocalStraightCorner}

In this section, we assume that the boundary $\partial_+\Omega$ satisfies the following {\bf straight corner condition}:
\begin{equation}\label{eq:OmegaCornerCondition}
\begin{cases}
f'(x)=f'(0)>0,\quad\mbox{for}\quad x\in[0, 3/4],\\
\Omega\subset \set{(x,y)| 0<y< f'(0)x,~x>0},
\end{cases}
\end{equation}
and 
\begin{equation*}
G=0\quad\mbox{on}\quad\partial_0\Omega.
\end{equation*}
Thus the boundary condition on $\partial_0\Omega$ in \eqref{eq:PotentialEquations} becomes Neumann boundary condition. Then we have the following theorem.
\begin{thm}\label{thm:LocalEstimate}
Let $u$ be the solution of \eqref{eq:PotentialEquations}. Then there exists $\sigma_0>0$ such that for any $\sigma\in (0,\sigma_0]$, one has
\begin{equation}\label{eq:LocalEstimate}
\norm{u}_{C^{2, \gamma}(\overline \Omega\cap\set{x\in[0, 2/3]})}\le \bar{C}\norm{\varphi}_{C^{2, \gamma}([0,1])},
\end{equation}
where  the two constants $\sigma_0,\bar C$ depend only on $\gamma,\lambda,\Lambda$.
\end{thm}

Set 
\begin{equation}\label{eq:v}
v=u-\varphi.
\end{equation}
Then $v$ solves 
\begin{equation}\label{eq:GeneralDirichletCondition0}
\begin{split}
& \mathcal L v=-\mathcal L(\varphi)=-A(x, y)\varphi''(x)-D(x, y)\varphi'(x),\quad \mbox{in}\quad \Omega,\\
& v=0,\quad \mbox{on}\quad \partial_+\Omega,\\
& v_y=0,\quad \mbox{on}\quad \partial_0\Omega,
\end{split}
\end{equation}
and the estimate
\begin{equation}\label{eq:LocalEstimatev}
\norm{v}_{C^{2, \gamma}(\overline \Omega\cap\set{x\in[0, 2/3]})}\le \bar{C}\norm{\varphi}_{C^{2, \gamma}([0,1])},
\end{equation}
induces \eqref{eq:LocalEstimate}.

To show above, for fixed $x_0\in(0, 2/3]$, we let $P_{x_0}(x, y)$ be the following quadratic polynomial
\begin{equation}\label{eq:px0}
\begin{split}
&P_{x_0}(x,y):=Np_{x_0}(x,y)\\
&:=N((f(x_0)^2-y^2)+(f^2)'(x_0)(x-x_0)+\frac 12(f^2)''(x_0)(x-x_0)^2)
\end{split}
\end{equation}
where $N=-\frac{\mathcal L (\varphi)(x_0, f(x_0))}{\mathcal L(p_{x_0})(x_0, f(x_0))}$.  Then 
\begin{equation*}
\begin{split}
-\mathcal LP_{x_0}(x_0, f(x_0))&=-\mathcal L(\varphi)(x_0, f(x_0))\\
&=-A(x_0, f(x_0))\varphi''(x_0)-D(x_0, f(x_0))\varphi'(x_0),
\end{split}
\end{equation*}
and
\begin{equation}\label{eq:Lpx0}
\begin{split}
\mathcal L(p_{x_0})&=2A(x, y)(f''(x_0)f(x_0)+(f'(x_0))^2)-2C(x, y)\\
&\quad+D(x, y)\qnt{2f'(x_0)f(x_0)+2(f''(x_0)f(x_0)+(f'(x_0))^2)(x-x_0)}\\
&\quad-2E(x, y)y\\
&\in [-10\Lambda, -\lambda]
\end{split}
\end{equation}
provided $(x, y)\in\Omega$ and $\sigma_0<<1$.
%  We can directly calculate
% \begin{equation}\label{eq:pExpressionD}
% \begin{cases}
% p_{x_0}(x_0, f(x_0))=\frac{\dif p_{x_0}(x, f(x))}{\dif x}|_{x=x_0}=\frac{\dif^2 p_{x_0}(x, f(x))}{\dif x^2}|_{x=x_0}=0,\\
% \frac{\dif^2 p_{x_0}(x, f(x))}{\dif x^2}=2(f''(x_0)f(x_0)+(f'(x_0))^2)-2(f''(x)f(x)+(f'(x))^2),\\
% \bracket{\frac{\dif^2 p_{x_0}(x, f(x))}{\dif x^2}}_{C^\gamma([0, 1])}\leq \bar{C}\norm{f}_{C^{2, \gamma}([0, 1])}^2.
% \end{cases}
% \end{equation}
By Taylor's expansion, one knows 
\begin{equation}\label{eq:PfEstimates}
\begin{split}
\abs{\qnt{v-P_{x_0}}(x, f(x))}&=\abs{P_{x_0}(x, f(x))}\\
&\le \bar C\|\varphi\|_{C^{2,\gamma}([0,1])}\norm{f}^2_{C^{2,\gamma}([0, 1])}|x-x_0|^{2+\gamma},
\end{split}
\end{equation}
and
\begin{equation}\label{eq:LPvEstimates}
\abs{\mathcal L\qnt{v-P_{x_0}}(x, y)}\le \bar C\norm{\varphi}_{C^{2, \gamma}([0, 1])}\abs{(x, y)-(x_0, f(x_0))}^{\gamma}.
\end{equation}

Denote
\begin{equation*}
\begin{split}
&\Omega^\kappa_{x_0}:=\set{(x, y)|~x_0-\kappa f(x_0)<x<x_0+\kappa f(x_0),~ 0<y< f(x)},\\
&\partial_+\Omega^\kappa_{x_0}:=\partial_+\Omega\cap\overline{\Omega^\kappa_{x_0}},\quad\partial_0\Omega^\kappa_{x_0}:=\partial_0\Omega\cap\overline{\Omega^\kappa_{x_0}}.
\end{split}
\end{equation*}
The following estimate holds. 
\begin{lem}\label{lem:C2HEstimates}
For fixed $x_0\in[0, 2/3]$, let $u$ be a solution of \eqref{eq:PotentialEquations} and $v, P_{x_0}$ be given as above. Then there exists $\sigma_0>0$ such that for any $\sigma\in (0,\sigma_0]$, one has 
\begin{equation}\label{eq:C2HEstimates}
\|v-P_{x_0}\|_{L^\infty(\Omega^1_{x_0})}\le  \bar C\norm{\varphi}_{C^{2,\gamma}([0, 1])}f(x_0)^{2+\gamma}
\end{equation}
where the two positive constants $\sigma_0,\bar C$ depend only on $\gamma,\lambda,\Lambda$. 
\end{lem}

%$\ell=\set{(x,y)|y=f'(0)x} $ is the tangent line of $\partial \Omega$ at $(x_0,y_0)$. Then $O=(0, 0)$ is a point on $\ell$. 

\begin{proof}
Introduce the polar coordinate
\begin{equation*}
r(x, y)=\sqrt{x^2+y^2},\quad \theta(x, y)=\arctan\frac{y}{x}. 
\end{equation*}
Consider the corner barrier function $Y$ as follows
\begin{equation*}
Y:=r^{-\alpha}\cos\qnt{k\theta}.
\end{equation*}
$\alpha, k$ are two constants which will be fixed later. Since $Y$ is even in $y$, one has
\begin{equation}\label{eq:YyVanish}
\begin{split}
Y_y(x,0)=0.
\end{split}
\end{equation}
A direct computation yields that
\begin{equation}\label{eq:LYExpression}
\begin{split}
\mathcal LY&=r^{-\alpha-2}\bigg(\alpha(\alpha+1)\bracket{A\cos^2\theta+B\sin\theta\cos\theta+C\sin^2\theta}\cos(k\theta)\\
&\quad+\alpha k\bracket{(C-A)\sin2\theta+B\cos2\theta}\sin(k\theta)\\
&\quad- k^2\bracket{A\sin^2\theta-B\sin\theta\cos\theta+C\cos^2\theta}\cos(k\theta)\\
&\quad- \alpha\bracket{(A\sin^2\theta-B\sin\theta\cos\theta+C\cos^2\theta)+Dr\cos\theta+Er\sin\theta}\cos(k\theta)\\
&\quad-k\bracket{(A-C)\sin2\theta+B\cos2\theta-Dr\sin\theta+Er\cos\theta}\sin(k\theta)\bigg).
\end{split}
\end{equation}

Let 
\begin{equation}\label{eq:kalphaDefinition}
k=\frac{\pi}{8f'(0)},\quad\mbox{and}\quad\alpha=\pm\frac{k}{M},
\end{equation}
for some positive constant $M$ large to be determined later. 
Then we have 
\begin{equation*}
\begin{split}
\mathcal LY&\leq r^{-\alpha-2}\qnt{\frac{2k^2}{M^2}2\Lambda+\frac{2k^2}{M}\Lambda- k^2\lambda\frac{1}{4}+\frac{2k}{M}\Lambda+k\Lambda+10kB+\frac{10kB}{M}},
\end{split}
\end{equation*}
provided $0<f'(0)\leq\sigma_0<<1$ and $\abs{\theta}\leq f'(0)$.

Let $M=\frac{100\Lambda}{\lambda}$. Then we have
\begin{equation}\label{eq:LYEstimates}
-\mathcal LY\geq \frac{\lambda k^2}{8}r^{-\alpha-2},
\end{equation}
provided $k>>1$.

Denote
\begin{equation}\label{eq:r0}
r_0=r(x_0, f(x_0))=\sqrt{1+\qnt{f'(0)}^2}\frac{f(x_0)}{f'(0)}\approx \frac{f(x_0)}{f'(0)}.
\end{equation}
Introduce barrier function $H$ as follows
\begin{equation*}
H:=f(x_0)^{2+\gamma}\qnt{\qnt{\frac{r}{r_0}}^{|\alpha|}+\qnt{\frac{r}{r_0}}^{-|\alpha|}}\cos\qnt{k\theta}.
\end{equation*} 
Then, by \eqref{eq:LYEstimates} and \eqref{eq:YyVanish}, we have
\begin{equation}\label{eq:LHEstimates}
\begin{split}
-\mathcal LH&\geq f(x_0)^{2+\gamma}\frac{\lambda k^2}{8}r_0^{-2}\qnt{\qnt{\frac{r}{r_0}}^{|\alpha|-2}+\qnt{\frac{r}{r_0}}^{-|\alpha|-2}}\\
&\geq f(x_0)^{\gamma}\frac{\lambda }{16}\qnt{\qnt{\frac{r}{r_0}}^{|\alpha|-2}+\qnt{\frac{r}{r_0}}^{-|\alpha|-2}},
\end{split}
\end{equation}
and
\begin{equation*}
H_y(x, 0)=0.
\end{equation*}

First, we estimate $-\mathcal LH$.
\begin{itemize}
\item[(1).] $\abs{x-x_0} \le f(x_0)$. In this case, by \eqref{eq:fEstimates} and Taylor's expansion, one knows 
\begin{equation*}
\abs{\frac{r}{r_0}-1}\le \sigma_0.
\end{equation*}
So \eqref{eq:LHEstimates} leads to
\begin{equation*}
\begin{split}
-\mathcal L H&\geq \frac{\lambda }{16}y_0^{\gamma}\left(\frac{r}{r_0}\right)^{-2}\bracket{\left(\frac{r}{r_0}\right)^{\alpha}+\left(\frac{r}{r_0}\right)^{-\alpha}}\\
&\geq \frac{\lambda }{16}y_0^{\gamma} \geq \bar C\abs{(x, y)-(x_0, f(x_0))}^{\gamma}.
\end{split}
\end{equation*}
\item[(2).] $f(x_0)\le x-x_0 $. In this case, by \eqref{eq:fEstimates} and Taylor's  expansion, one knows 
\begin{equation*}
r-r_0\geq\frac{1}{2} (x-x_0)
\end{equation*}
then
\begin{equation*}
\frac{r}{r_0}\ge 1+\frac{1}{2} \frac{x-x_0}{r_0}.
\end{equation*}
Hence, by \eqref{eq:kalphaDefinition} and \eqref{eq:r0}, we have
\begin{equation*}
\begin{split}
\left(\frac{r}{r_0}\right)^{\qnt{\frac{k}{M}-2}\frac{1}{\gamma}}&\ge\left (1+\frac{1}{2} \frac{x-x_0}{r_0}\right)^{\qnt{\frac{k}{M}-2}\frac{1}{\gamma}}\\
&\ge 1+\frac{1}{2} \frac{x-x_0}{r_0}\qnt{\frac{k}{M}-2}\frac{1}{\gamma}\\
&\geq 1+\delta_0\frac{x-x_0}{f(x_0)}
\end{split}
\end{equation*}
for a positive constant $\delta_0$ depend on $\lambda, \Lambda$ and $\gamma$. Hence
\begin{equation*}
\begin{split}
-\mathcal LH\ge \bar C f(x_0)^\gamma\qnt{1+\delta_0 \frac{x-x_0}{f(x_0)}}^\gamma\geq \bar C f(x_0)^\gamma\qnt{\delta_0 \frac{x-x_0}{f(x_0)}}^\gamma\ge \bar C |x-x_0|^\gamma. 
\end{split}
\end{equation*}
\item[(3).] For the case $-f(x_0)\geq x-x_0 $, as above, we have
\begin{equation*}
r-r_0\leq \frac{1}{2} (x-x_0)
\end{equation*}
then
\begin{equation*}
\frac{r}{r_0}\leq 1+\frac{1}{2} \frac{x-x_0}{r_0}<1.
\end{equation*}
Hence
\begin{equation*}
\begin{split}
\left(\frac{r}{r_0}\right)^{\qnt{-\frac{k}{M}-2}\frac{1}{\gamma}}&\ge\left (1+\frac{1}{2} \frac{x-x_0}{r_0}\right)^{\qnt{-\frac{k}{M}-2}\frac{1}{\gamma}}\\
&\ge 1+\frac{1}{2} \frac{x-x_0}{r_0}\qnt{-\frac{k}{M}-2}\frac{1}{\gamma}\\
&\geq 1+\delta_0 \frac{x_0-x}{f(x_0)},
\end{split}
\end{equation*}
and
\begin{equation*}
-\mathcal LH\ge \bar C f(x_0)^\gamma\left(1+\delta_0 \frac{x_0-x}{f(x_0)}\right)^\gamma\geq \bar C f(x_0)^\gamma\left(\delta_0 \frac{x_0-x}{f(x_0)}\right)^\gamma\ge \bar C |x-x_0|^\gamma. 
\end{equation*}
\end{itemize}

Similar arguments also imply on $\partial_+\Omega$
\begin{equation*}
H\ge \bar C |x-x_0|^{2+\gamma}.
\end{equation*}

By \eqref{eq:GeneralDirichletCondition0}, \eqref{eq:PfEstimates} and \eqref{eq:LPvEstimates}, applying standard maximum principle, we have
\begin{equation}\label{eq:vPH}
\abs{v(x, y)-P_{x_0}(x, y)}\leq \bar C\norm{\varphi}_{C^{2, \gamma}([0,1])}H(x, y).
\end{equation}
For $(x, y)\in\Omega_{x_0}^1$, we have
\begin{equation*}
\abs{\frac{r}{r_0}-1}\leq 3f'(0).
\end{equation*}
Then
\begin{equation*}
\begin{split}
H&\leq f(x_0)^{2+\gamma}\abs{\left(\frac{r}{r_0}\right)^{|\alpha|}+\left(\frac{r}{r_0}\right)^{-|\alpha|}}\\
&\leq 2f(x_0)^{2+\gamma}\qnt{1+3f'(0)}^{\frac{\pi}{8Mf'(0)}}\\
&\leq \bar C f(x_0)^{2+\gamma}.
\end{split}
\end{equation*}
Thus, by \eqref{eq:vPH}, we complete our proof.
\end{proof}

\begin{rem}\label{rm:DEcondition}
In the construction of barrier function $H$, the condition on the coefficients $D,E$ can be weaken to 
\begin{equation*}
r|D|+r|E|\le \Lambda.
\end{equation*}
\end{rem}
\begin{proof}[{\bf Proof of Theorem \ref{thm:LocalEstimate}}]
For $x_0\in[0, 2/3]$, consider the scaling domains as follows
\begin{equation}\label{eq:coordinatechange}
\begin{split}
&\widetilde {\Omega}^\kappa_{x_0}:=\set{(\tilde {x}, \tilde {y})|~(\tilde {x}, \tilde {y})=\frac{1}{ f(x_0)}(x-x_0, y),~(x, y)\in\Omega^\kappa_{x_0}},\\
&\partial_+\widetilde {\Omega}^\kappa_{x_0}:=\set{(\tilde {x}, \tilde {y})|~(\tilde {x}, \tilde {y})=\frac{1}{ f(x_0)}(x-x_0, y),~(x, y)\in\partial_+\Omega^\kappa_{x_0}},\\
&\partial_0\widetilde {\Omega}^\kappa_{x_0}:=\set{(\tilde {x}, \tilde {y})|~(\tilde {x}, \tilde {y})=\frac{1}{ f(x_0)}(x-x_0, y),~(x, y)\in\partial_0\Omega^\kappa_{x_0}}.
\end{split}
\end{equation}
Denote 
$$w(x,y)=(v-P_{x_0})(x,y),\quad \tilde w(\tilde x,\tilde y)=\frac{w(x,y)}{(f(x_0))^{2+\gamma}}.$$
Then $\tilde w$ solves 
\begin{equation*}
\begin{split}
\tilde {\mathcal L}\tilde w=-\frac{\mathcal L\varphi(x, y)-\mathcal L\varphi(x_0, y_0)}{(f(x_0))^\gamma}=:\tilde g(\tilde x, \tilde y)
\end{split}
\end{equation*}
where 
\begin{equation*}
\tilde {\mathcal L}=A\partial_{\tilde x\tilde x}+B\partial_{\tilde x\tilde y}+C\partial_{\tilde y\tilde y}+f(x_0)(D\partial_{\tilde x}+E\partial_{\tilde y}).
\end{equation*}
Since 
\begin{equation*}
(-1,1)\times (-1/2,1/2)\subset \tilde \Omega_{x_0}^1\subset (-1,1)\times (0,2),\quad \|\partial_+ \Omega_{x_0}^1\|_{C^{2,\gamma}}\le \bar C,
\end{equation*} 
by the standard  Schauder estimates, we have  that
\begin{equation}\label{eq:HEstimateStandard}
\begin{split}
\norm{\tilde w}_{C^{2, \gamma}(\widetilde {\Omega}^{1/2}_{x_0})}\leq \bar C\Big(\|\tilde g\|_{C^\gamma(\widetilde {\Omega}^1_{x_0})}+\norm{\frac{P_{x_0}}{(f(x_0))^{2+\gamma}}}_{C^{2, \gamma}(\partial_+\widetilde {\Omega}^1_{x_0})}+\norm{\tilde w}_{C^0(\widetilde {\Omega}^1_{x_0})}\Big),
\end{split}
\end{equation}
where the constant $\bar C$ depends only on $\gamma,\lambda,\Lambda$. It is easy to see that 
\begin{equation*}
\|\tilde g\|_{C^\gamma(\widetilde {\Omega}^1_{x_0})}\le \bar C \|\varphi\|_{C^{2,\gamma}([0,1])}.
\end{equation*}
By \eqref{eq:PfEstimates}, one knows 
\begin{equation*}
\norm{\frac{P_{x_0}}{(f(x_0))^{2+\gamma}}}_{L^\infty(\partial_+\widetilde {\Omega}^1_{x_0})}\le \bar C\|\varphi\|_{C^{2,\gamma}([0,1])}.
\end{equation*}
By a direct computation, one also has
\begin{equation*}
\frac{\dif^2}{\dif x^2}P_{x_0}(x,f(x))=-N(f'(x)^2-f'(x_0)^2+f''(x)f(x)-f''(x_0)f(x_0))\in C^{\gamma}([0,1]).
\end{equation*}
By Lemma \ref{lem:C2HEstimates} and the above estimates, \eqref{eq:HEstimateStandard} implies  
\begin{equation*}
\norm{\tilde w}_{C^{2, \gamma}(\widetilde {\Omega}^{1/2}_{x_0})}\leq \bar{C}\|\varphi\|_{C^{2,\gamma}([0,1])}.
\end{equation*}
In the coordinate $(x, y)$, since $x_0$ can be arbitrary in $(0,2/3)$, above implies 
\begin{equation*}
\begin{split}
\sup_{\substack{(x_1,y_1),(x_2,y_2)\in \bar \Omega\cap \{x\in[0,2/3]\}\\
|x_1-x_2|\le \frac 12\max(f(x_1),f(x_2))}} &\frac{|D^2 u(x_1,y_1)-D^2 u(x_2,y_2)|}{|(x_1,y_1)-(x_2,y_2)|^{\gamma}}\\
+&\|u\|_{C^2(\bar \Omega\cap \{x\in [0,2/3]\})}\le  \bar{C}\|\varphi\|_{C^{2,\gamma}([0,1])}.
\end{split}
\end{equation*}
Consider the following two points
\begin{equation*}
(x_1,y_1),(x_2,y_2)\in \bar \Omega\cap \{x\in [0,2/3]\},\quad |x_1-x_2|\ge \frac 12\max(f(x_1),f(x_2)).
\end{equation*}
  Since $u_y(x,0)=0$, there holds 
\begin{equation*}
\begin{split}
&|u_{xy}(x_1,y_1)-u_{xy}(x_2,y_2)|\\
\le &|u_{xy}(x_1,0)-u_{xy}(x_1,y_1)|+|u_{xy}(x_2,0)-u_{xy}(x_2,y_2)|+|u_{xy}(x_1,0)-u_{xy}(x_2,0)|\\
\le & \bar C\|\varphi\|_{C^{2,\gamma}([0,1])}(y_1^\gamma+y_2^\gamma)\le \bar C\|\varphi\|_{C^{2,\gamma}([0,1])}|(x_1,y_1)-(x_2,y_2)|^\gamma.
\end{split}
\end{equation*}
This implies 
\begin{equation*}
\|u_{xy}\|_{C^{2,\gamma}(\bar \Omega\cap \{x\in [0,2/3]\})}\le \bar{C}\|\varphi\|_{C^{2,\gamma}([0,1])}.
\end{equation*}

Notice that
\begin{equation*}
\begin{split}
u_{xx}+(f'(x))^2u_{yy}=&\frac{\dif^2 u(x,f(x))}{\dif x^2}-2f'(x)u_{xy}-f''(x)u_y\\
=&\frac{\dif^2 \varphi(x)}{\dif x^2}-2f'(x)u_{xy}-f''(x)u_y,
\end{split}
\end{equation*}
and \eqref{eq:PotentialEquations}
\begin{equation*}
Au_{xx}+Cu_{yy}=-Bu_{xy}-Du_x-Eu_y.
\end{equation*}
Hence for $\sigma_0$ small, we can solve $u_{xx}$ and $u_{yy}$ from above two to have
\begin{equation*}
\|u_{xx}\|_{C^{\gamma}(\partial_+\Omega\cap \{0\le x\le 2/3\})}+\|u_{yy}\|_{C^{\gamma}(\partial_+\Omega\cap \{0\le x\le 2/3\})}\le \bar C\|\varphi\|_{C^{2,\gamma}([0,1])}.
\end{equation*}
Then a similar argument as $u_{xy}$ yields that 
\begin{equation*}
\|u_{xx}\|_{C^{\gamma}(\bar\Omega\cap \{0\le x\le 2/3\})}+\|u_{yy}\|_{C^{\gamma}(\bar \Omega\cap \{0\le x\le 2/3\})}\le \bar C\|\varphi\|_{C^{2,\gamma}([0,1])}.
\end{equation*}
This ends the proof for Theorem \ref{thm:LocalEstimate}.
\end{proof}

\section{General Case}
In this section, we remove the straight corner condition and the restriction $G(x)=0$ to complete the proof of Theorem \ref{thm:GlobalEstimates}.

\begin{comment}
\begin{split}
&\norm{\frac{\chi(x) \sigma_0 x+(1-\chi(x))f(x)}{f(x)}}_{C^{2, \gamma}([0, 1])}\\
&\quad+\norm{\frac{\chi(x) \sigma_0 x+(1-\chi(x))f(1-x)}{f(1-x)}}_{C^{2, \gamma}([0, 1])}\leq C_f.
\end{split}
\end{comment}

\begin{proof}[Proof of Theorem \ref{thm:GlobalEstimates}]
Theorem \ref{thm:GlobalEstimates} can be reduced to the situation discussed in Theorem \ref{thm:LocalEstimate} through three coordinate transformations.

(1) Change the oblique derivative condition to Neumann boundary condition. Define the trajectory $x(y;s)$ by  
\begin{equation*}
\begin{cases}
\frac{\dif x(y; s)}{\dif y}=G(x),\\
x(0; s)=s,
\end{cases}
\end{equation*}
Introduce a new coordinate $(s,z)$ by
\begin{equation}\label{transformp1}
\mathcal{P}_1:
\begin{cases}
x=x(z;s),\\
y=z.
\end{cases}
\end{equation}
Since
\begin{equation*}
\begin{pmatrix}
x_s& x_z\\
y_s& y_z
\end{pmatrix}\Bigg|_{y=0}
=
\begin{pmatrix}
1& G(x)\\
0&1
\end{pmatrix},
\end{equation*}
for $\sigma_0$ small, the map 
\begin{equation*}
\begin{split}
\mathcal P_1^{-1}:\	&\	\	\overline{\Omega}\	\	\	\rightarrow \mathcal P_1^{-1}(\overline{\Omega})\\
 &(x,y)\rightarrow \	(s,z)
\end{split}
\end{equation*}
is a $C^{2,\gamma}$-diffeomorphism around $[0, 1]\times\set{0}$ and $\mathcal P_1^{-1}([0,1]\times \{0\})=[0,1]\times \{0\}.$
It remains to consider the image of $\partial_+\Omega$, i.e. $(s,f(x(z;s)))$. 
By the coordinate transformation, we have 
\begin{equation*}
s=x(z;s)-\int_0^z G(x(\lambda;s))d\lambda.
\end{equation*}
Hence on the image of $\partial_+\Omega$, we have
\begin{equation*}
s=x-\int_0^{f(x)} G(x(\lambda;s))d\lambda.
\end{equation*}
This implies $s'(x)=1-f'(x)G(x)>0$ for $x\in [0,1]$ provided $\sigma_0$ is small. From this, we know $x\rightarrow s$ is a $C^{2,\gamma}$-diffeomorphism from $[0,1]$ to $[0,1]$, i.e. the image  $\mathcal P_1^{-1}(\partial_+\Omega)$ can be written by 
\begin{equation*}
z=f(x(z;s))=\hat f(s).
\end{equation*}
By a direct computation, the  oblique derivative condition of \eqref{eq:PotentialEquations} becomes
\begin{equation*}
u_z=0,~\text{on}~z=0.
\end{equation*}
Since the coordinate transformation $\mathcal{P}_1$ is a diffeomorphism, we can directly verify that condition \eqref{eq:FiniteGrowCondition} is preserved, where the constant $C_f$ is updated to the new constant $\hat{C}_f$. Then, the estimates in the new coordinates naturally induce estimate \eqref{eq:ObliqueDerivativeSchauder}. Therefore, we only need to prove the conclusion for the case when $G=0$.

(2) Straighten the boundary $\partial_+\Omega=\set{(x, f(x))|x\in [0,1]}$ under the assumption $G=0$. Consider a cut off function $\chi\in C^\infty([0, 1])\rightarrow [0,1]$ with
\begin{equation}\label{eq:chiSetting}
\chi=\begin{cases}
1,~x\in[0, 3/4],\\
0,~x\in[13/16, 1],
\end{cases}
\end{equation}
\begin{comment}
Since \eqref{eq:FiniteGrowCondition}, by the argument in (1), we can assume 
\begin{equation}\label{eq:FiniteGrowCondition-sz}
\inf_{s\in [0,1]}\left|\frac{f(s)}{\sin (\pi s)}\right|=\kappa,\	 \|f\|_{C^2([0,1])}=C_{\hat f} \kappa =\overline{\kappa}
\end{equation}
for some $\kappa>0$ small and $C_{\hat f}>1$ positive and uniformly bounded. 
\end{comment}
and introduce the following coordinate transformation:
\begin{equation}\mathcal{P}_2:
\begin{cases}
s=x,\\
z=\frac{\chi(x) \bar{\Pi}x+(1-\chi(x))f(x)}{f(x)}y.
\end{cases}
\end{equation}
This transformation reduces the domain $\Omega$ to a straight corner domain $\mathcal P_2(\Omega)$.

(3) To get the estimate for $x\in[1/3, 1]$, we can use the coordinate transformation
\begin{equation}\mathcal{P}_3:
\begin{cases}
s=1-x,\\
z=y.
\end{cases}
\end{equation}
Thus the corner $(x, y)=(1, 0)$ is reduced to the origin $O$.

In the next, we explain how to use the above three coordinate transformations to get Theorem \ref{thm:GlobalEstimates}. 
\par The transformation $\mathcal{P}_1$ is locally determined by $G$, independent of $f$, and smooth. Thus, the estimate under the new coordinates induces the original estimate.  $\mathcal{P}_3$ is the same. The only trouble thing is the transformation $\mathcal{P}_2$. 
\par For $x_0\in \partial\Omega\cap (0, 2/3]$, let $v, P_{x_0}$ be given by \eqref{eq:v}, \eqref{eq:px0} respectively. We first need to show that the estimate \eqref{eq:C2HEstimates} holds. Denote the elliptic operator in the original coordinate $(x,y)$ by $\mathcal L_0$, in $(s,z)$ by $\mathcal L_1$. It can be checked directly that estimate \eqref{eq:PfEstimates} and \eqref{eq:LPvEstimates} hold in  $(s,z)-$coordinate. 
%By the discussion in (1), the estimate \eqref{eq:PfEstimates} and \eqref{eq:LPvEstimates}  also hold in  $(x,y)-$coordinate. 
In fact, we need compute the Jacobi matrix for the coordinate transformation $\mathcal P_2$.

Denote
\begin{equation}\label{eq:f1Expression}
f_1(x):=\chi(x) \bar{\Pi}x+(1-\chi(x))f(x),
\end{equation}
we have
\begin{equation}\label{eq:f1ExpressionC}
\frac{f_1}{f}=1,~\mbox{for}~x\in[13/16, 1].
\end{equation}

Then
\begin{equation}\label{eq:jacobixysz}
\begin{pmatrix}
s_x&s_y\\
z_x&z_y
\end{pmatrix}
=
\begin{pmatrix}
1&0\\
\frac{f_1'f-f_1f'}{f^2}y&\frac{f_1}{f}
\end{pmatrix},
\end{equation}
\begin{equation*}
\begin{cases}
s_{xx}=s_{xy}=s_{yy}=z_{yy}=0,~z_{xy}=\frac{f_1'f-f_1f'}{f^2},\\
z_{xx}=\frac{f_1''f-f_1f''}{f^2}y+2\frac{f_1'f'f-f_1\qnt{f'}^2}{f^3}y,
\end{cases}
\end{equation*}
and
\begin{equation*}
\mathcal L_1 u=A_1u_{ss}+B_1 u_{sz}+C_1 u_{zz}+D_1 u_{s}+E_1 u_{z},
\end{equation*}
where 
\begin{equation*}
\begin{split}
& A_1=A s_x^2+Bs_xs_y+Cs_y^2,\	B_1=2As_xz_x+Bs_yz_x+Bs_xz_y+2Cs_yz_y,\\
& C_1=Az_x^2+Bz_yz_x+Cz_y^2,\	D_1=As_{xx}+Bs_{xy}+Cs_{yy}+Ds_x+Es_y,\\
& E_1=Az_{xx}+Bz_{xy}+Cz_{yy}+Dz_x+Ez_y.
\end{split}
\end{equation*}

Now we estimate the Jacobi matrix \eqref{eq:jacobixysz}. By \eqref{eq:FiniteGrowCondition}, we know
\begin{equation*}
\begin{split}
\frac{f_1}{f}=&(1-\chi(x))+\chi(x)\frac{\bar{\Pi}x}{f(x)}\\
\le & 1+C_{f},
\end{split}
\end{equation*}
and also 
\begin{equation*}
\begin{split}
\frac{f_1}{f}\ge \min \left(1, C_f\right)=1.
\end{split}
\end{equation*}
Moreover, since $y\leq f$, there holds
\begin{equation*}
\begin{split}
\left|\frac{f_1'f-f_1f'}{f^2}y\right|\le  |f_1'|+\left|\frac{f_1f'}{f}\right|
\le  C\bar \Pi+\bar \Pi \left|\frac{f_1}{f}\right|\le \bar C(1+C_f)^2 \bar \Pi<<1.
\end{split}
\end{equation*}
Above three imply 
$\begin{pmatrix}
s_x&s_y\\
z_x&z_y
\end{pmatrix}\approx \begin{pmatrix}
*&0\\
0&*
\end{pmatrix}
$ for $\bar \Pi$ small enough, i.e. $\mathcal P_2$ is a Lipschitz diffeomorphism. Hence the estimates \eqref{eq:PfEstimates} and \eqref{eq:LPvEstimates}  also hold in  $(s, z)-$coordinate.
\par The key to the proof of Lemma \ref{lem:C2HEstimates} is the existence of barrier function $H$. By Remark \ref{rm:DEcondition}, we only need to check that $\sqrt{s^2+z^2} |D_1|$ and $\sqrt{s^2+z^2} |E_1|$ are uniformly bounded. It is enough to estimate the following two terms
\begin{equation*}
\begin{split}
\left|\sqrt{s^2+z^2}\frac{f_1' f-f_1 f'}{ f^2} \right|&=\left|\sqrt{s^2+z^2}\qnt{\frac{f_1}{ f}}' \right|\\
&\le \bar{C} \left|s\frac{f_1' f-f_1 f'}{ f^2} \right|~\mbox{for}~s\in[0,3/4]\\
&=\bar{C} \left|x\frac{f_1' f-f_1 f'}{ f^2} \right|~\mbox{for}~x\in[0,3/4]\\
&\le \bar C(1+C_f)^2,
\end{split}
\end{equation*}
and
\begin{equation*}
\begin{split}
& \left|\sqrt{s^2+z^2} \left(\frac{f_1''f-f_1f''}{f^2}y-2\frac{f_1'f'f-f_1\qnt{f'}^2}{f^3}y\right)\right|\\
\le &\bar{C}\frac{s}{f}\left(|f_1''| f+|f''|f_1+|f_1'f'|+|f'|^2 \frac{f_1}{f}\right)~\mbox{for}~s\in[0,3/4]\\ \le & \bar C(1+C_f)^2,
\end{split}
\end{equation*}
where we utilize \eqref{eq:f1ExpressionC}. Above two imply $\sqrt{s^2+z^2} |D_1|$ and $\sqrt{s^2+z^2} |E_1|$ are uniformly bounded. 

From the above discussion, we can construct the barrier function $H$ as in Lemma \ref{lem:C2HEstimates} to get the desired estimate
\begin{equation*}
|(v-P_{x_0})(s_0, z_0)|\le \bar C\norm{\varphi}_{C^{2,\gamma}([0, 1])}f(s_0)^{2+\gamma},\quad \mbox{for}~s_0\in [0, 2/3].
\end{equation*}
The above estimate will give the corresponding growth estimate in the original coordinate $(x, y)$
\begin{equation}\label{eq:wgrowth2}
\|(v-P_{x_0})\|_{L^\infty(\Omega_{x_0}^{1/2})}\le \bar C\norm{\varphi}_{C^{2,\gamma}([0, 1])} f(x_0)^{2+\gamma},\quad \mbox{for}~  x_0\in [0,2/3].
\end{equation}
Once we have estimate \eqref{eq:wgrowth2}, we do the same argument as in Theorem \ref{thm:LocalEstimate} in $(x, y)$-coordinate. The proof is complete. 
\end{proof}

\section{Unbounded Dirichlet Condition}
In this section, we extend Theorem \ref{thm:GlobalEstimates} to the weighted Hölder spaces. For $m>0$, define the weighted H\"{o}lder norms as follows
\begin{equation*}
\begin{split}
\norm{\varphi}_{C^{(-m)}_{k, \gamma}([0, 1])}&=\max\limits_{i=0,...,k}\sup\limits_{x\in[0, 1]}\abs{x^{m+i}\varphi^{(i)}(x)}\\
&\quad+\sup\limits_{0\leq x_1< x_2\leq1}\abs{x_1^{m+k+\gamma}\frac{\varphi^{(k)}(x_1)-\varphi^{(k)}(x_2)}{\abs{x_1-x_2}^\gamma}},
\end{split}
\end{equation*}
\begin{equation*}
\begin{split}
\norm{u}_{C^{(-m)}_{k, \gamma}(\overline{\Omega})}&=\max\limits_{i=0,...,k}\sup\limits_{(x, y)\in\Omega}\abs{r^{m+i}\nabla^{i} u(x, y)}\\
&\quad+\sup\limits_{0\leq r_1\leq r_2\leq1}\abs{r_1^{m+k+\gamma}\frac{\nabla^{k} u(x_1, y_1)-\nabla^{k} u(x_2, y_2)}{\abs{(x_1, y_1)-(x_2, y_2)}^\gamma}},
\end{split}
\end{equation*}
where $r=\sqrt{x^2+y^2}$, $r_i=\sqrt{x_i^2+y_i^2}$, $i=1,2$, and
\begin{equation*}
\begin{split}
\norm{f}_{C^{(1+\gamma)}_{2, \gamma}([0, 1])}&=\max\limits_{i=0, 1, 2}\sup\limits_{x\in[0, 1]}\abs{xf^{(i)}(x)}+\sup\limits_{0\leq x_1< x_2\leq1}\abs{x_1\frac{f^{(2)}(x_1)-f^{(2)}(x_2)}{\abs{x_1-x_2}^\gamma}}.
\end{split}
\end{equation*}
From Lemma 2.1 in \cite{GilbargHormander}, one also knows 
\begin{equation}\label{eq:h-embed}
\sup_{x\in[0,1]}|x^{1-\gamma}f''(x)|+\|f\|_{C^{1,\gamma}([0,1])}\le \bar{C}\norm{f}_{C^{(1+\gamma)}_{2, \gamma}([0, 1])}. 
\end{equation}
We also replace \eqref{eq:fEstimates} and \eqref{eq:FiniteGrowCondition} by the following conditions: 
\begin{equation}\label{eq:fEstimatesW}
\abs{f}_{C^{(1+\gamma)}_{2, \gamma}([0, 1])}\le \sigma,
\end{equation}
and
\begin{equation}\label{eq:FiniteGrowConditionW}
\begin{cases}
\inf_{x\in[0, 1]}\abs{\frac{f(x)}{\sin(\pi x)}}=\Pi>0,\\
\norm{f}_{C^{(1+\gamma)}_{2, \gamma}([0, 1])}=C_f\Pi=\bar{\Pi}.
\end{cases}
\end{equation}
Then, we have
\begin{thm}\label{thm:GlobalEstimatesW}
Let $u$ be the solution of \eqref{eq:PotentialEquations}. Assume that $A, B, C, D, E$ and $G$ satisfy \eqref{eq:uniformelliptic} and $f$ satisfies \eqref{eq:fEstimatesW} and \eqref{eq:FiniteGrowConditionW}. Then there exists $\sigma_0>0$ such that for any $ \sigma\in (0,\sigma_0]$, we have 
\begin{equation}\label{eq:ObliqueDerivativeSchauderW}
\norm{u}_{C_{2, \gamma}^{(-m)}(\Omega)}\le \bar{C}\norm{\varphi}_{C_{2, \gamma}^{(-m)}([0,1])},
\end{equation}
where the two positive constants $\sigma_0, \bar{C}$ depend only on $\gamma, \lambda, \Lambda, C_f, m$.
\end{thm}

Recall that $v=u-\varphi,P_{x_0}(x,y)=N p_{x_0}(x,y)$ are given by \eqref{eq:v} and \eqref{eq:px0}.
Since for $0<x_1<x_2<1$, from \eqref{eq:h-embed} and \eqref{eq:fEstimatesW}, we have 
\begin{equation}\label{eq:h001}
\begin{split}
&\frac{|f''(x_1)f(x_1)-f''(x_2)f(x_2)|}{|x_1-x_2|^\gamma}\\
\le & \frac{|f(x_1)|}{x_1}\left|x_1\frac{f''(x_1)-f''(x_2)}{|x_1-x_2|^\gamma}\right|+(x_2^{1-\gamma}|f''(x_2)|)\frac{|f(x_1)-f(x_2)|}{|x_1-x_2|}\frac{|x_1-x_2|^{1-\gamma}}{x_2^{1-\gamma}}\\
\le & \bar C\|f\|_{C^{(1+\gamma)}_{2,\gamma}}^2.
\end{split}
\end{equation}
Hence, by the definition of $p_{x_0}$, we have 
\begin{equation*}
\abs{p_{x_0}(x, f(x))}\leq\bar{C}\norm{f}_{C^{(1+\gamma)}_{2, \gamma}([0, 1])}^2\abs{x-x_0}^{2+\gamma}.
\end{equation*}
Also for $\sigma_0$ small enough, \eqref{eq:Lpx0} still holds in the present situation. 
Moreover, a simple calculation yields that 
\begin{equation*}
|\mathcal L(\varphi)(x_0, f(x_0))|\le \bar{C}x_0^{-\qnt{m+2}}\norm{\varphi}_{C^{(-m)}_{2, \gamma}([0, 1])}.
\end{equation*}
Then we have the estimates for $P_{x_0}(x, f(x))$ as follows
\begin{equation}\label{eq:PfEstimatesW}
\begin{split}
\abs{\qnt{v-P_{x_0}}(x, f(x))}&=\frac{\abs{\mathcal L(\varphi)(x_0, f(x_0))}}{\mathcal L(p_{x_0}(x_0,f(x_0)))}\abs{p_{x_0}(x, f(x))}\\
&\leq
\frac{\bar{C}\norm{\varphi}_{C^{(-m)}_{2, \gamma}([0, 1])}\norm{f}_{C^{(1+\gamma)}_{2, \gamma}([0, 1])}^2}{x_0^{m+2}}\abs{x-x_0}^{2+\gamma}.
\end{split}
\end{equation}
Similarly, we have 
\begin{equation}\label{eq:LPvEstimatesW}
\displaystyle
\begin{split}
\abs{\mathcal L(v-P_{x_0})}\leq 
\begin{cases}
\displaystyle\frac{\bar{C}\norm{\varphi}_{C^{(-m)}_{2, \gamma}([0, 1])}}{x^{m+2+\gamma}}|(x,y)-(x_0,y_0)|^{\gamma},~\mbox{for}~ x\leq\frac{x_0}{2};\\ \\
\displaystyle\frac{\bar{C}\norm{\varphi}_{C^{(-m)}_{2, \gamma}([0, 1])}}{x_0^{m+2+\gamma}}|(x, y)-(x_0, y_0)|^{\gamma},~\mbox{for}~ x\in[\frac{x_0}{2}, 1].
\end{cases}
\end{split}
\end{equation}

Our proof is similar to Theorem \ref{thm:GlobalEstimates}. We first prove that the estimate holds with the straight corner condition, and then obtain our theorem through three coordinate transformations.
\begin{lem}\label{lem:C2HEstimatesW}
Under the assumption of Theorem \ref{thm:GlobalEstimatesW}, there exists $\sigma_0>0$ such that for any $ \sigma\in (0,\sigma_0]$, we have 
\begin{equation}\label{eq:C2HEstimatesW}
\|v-P_{x_0}\|_{L^\infty(\Omega^1_{x_0})}\le  \bar{C}x_0^{-\qnt{m+2+\gamma}}\norm{\varphi}_{C^{(-m)}_{2, \gamma}([0, 1])}f(x_0)^{2+\gamma}
\end{equation}
where the two constants $\sigma_0,\bar{C}$ depend only on $\gamma, \lambda, \Lambda, C_f, m$.
\end{lem}
\begin{proof}
We first deal with {\bf straight corner case}.
Define the barrier function  $H$ with weight as follows
\begin{equation*}
H=x_0^{-\qnt{m+2+\gamma}}f(x_0)^{2+\gamma}\qnt{\qnt{\frac{r}{r_0}}^{|\alpha|}+\qnt{\frac{r}{r_0}}^{-|\alpha|}}\cos\qnt{k\theta}
\end{equation*}
where $\alpha,k$ is given by Lemma \ref{lem:C2HEstimates}.
The proof here is similar to Lemma \ref{lem:C2HEstimates}, one only needs to take into account the impact of the weight terms. Here, we briefly list the key steps. As \eqref{eq:LHEstimates},
\begin{equation*}
\begin{split}
-\mathcal LH&\geq x_0^{-\qnt{m+2+\gamma}} f(x_0)^{2+\gamma}\frac{\lambda k^2}{8}r_0^{-2}\qnt{\qnt{\frac{r}{r_0}}^{|\alpha|-2}+\qnt{\frac{r}{r_0}}^{-|\alpha|-2}}\\
&\geq x_0^{-\qnt{m+2+\gamma}} f(x_0)^\gamma\frac{\lambda }{16}\qnt{\qnt{\frac{r}{r_0}}^{|\alpha|-2}+\qnt{\frac{r}{r_0}}^{-|\alpha|-2}}.
\end{split}
\end{equation*}
Then the discussion in Lemma \ref{lem:C2HEstimates}  also implies 
\begin{equation*}
\begin{split}
-\mathcal L H  \geq \bar{C}x_0^{-\qnt{m+2+\gamma}}\abs{(x, y)-(x_0, f(x_0))}^\gamma,\quad x\ge \frac{x_0}2.
\end{split}
\end{equation*}
It remains to estimate the case $x\le \frac{x_0}{2}$. 
Since 
\begin{equation*}
r\approx x,\quad\mbox{and}\quad r_0\approx x_0,
\end{equation*}
we may assume $\frac{r_0}{r}\ge \frac 32$. By \eqref{eq:kalphaDefinition}, it follows that 
\begin{equation*}
\begin{split}
-\mathcal LH \ge &\bar{C}x_0^{-m-2-\gamma}(f(x_0))^\gamma \left(\frac{r_0}{r}\right)^{\frac{k}{2M}}\\
\ge & \bar{C}x_0^{-m-2-\gamma}\left(\frac{r_0}{r}\right)^{\frac k{4M}}x_0^\gamma \times (f'(x_0))^\gamma \left(\frac 32\right)^{\frac{k}{4M}}\\
\ge & \bar{C}x_0^{-m-2-\gamma}\left(\frac{r_0}{r}\right)^{\frac k{4M}}x_0^\gamma \times (f'(x_0))^\gamma \qnt{\left(\frac 32\right)^{\frac{k}{4\gamma M}}}^\gamma\\
\ge & \bar{C}x_0^{-m-2-\gamma}\left(\frac{r_0}{r}\right)^{\frac k{4M}}x_0^\gamma \times (f'(x_0))^\gamma \qnt{\frac{k}{8\gamma M}}^\gamma\\
\ge & \bar{C} x^{-(m+2+\gamma)}|(x, y)-(x_0, f(x_0))|^{\gamma}
\end{split}
\end{equation*}
for $\frac{k}{4M}\geq m+2+\gamma$ and $\sigma_0$ small enough.

Overall, we have proved that 
\begin{equation*}
\begin{split}
-\mathcal L(H)\geq 
\begin{cases}
\displaystyle\bar{C} x^{-(m+2+\gamma)}|(x,y)-(x_0, f(x_0))|^{\gamma},~\mbox{for}~ x\in \left[0,\frac{x_0}{2}\right];\\\\
\displaystyle\bar{C} x_0^{-(m+2+\gamma)}\abs{(x, y)-(x_0, f(x_0))}^{\gamma},~\mbox{for}~ x\in\left[\frac{x_0}{2}, 1\right].
\end{cases}
\end{split}
\end{equation*}
Similar arguments also implies on $\partial\Omega$
\begin{equation*}
H\geq\frac{\bar{C}}{x_0^{m+2+\gamma}}\abs{x-x_0}^{2+\gamma}.
\end{equation*}
As the proof of Lemma \ref{lem:C2HEstimates}, in view of \eqref{eq:PfEstimatesW} and \eqref{eq:LPvEstimatesW}, we deduce \eqref{eq:C2HEstimatesW} via above two.

For the {\bf general case}, we apply the coordinates transformation $\mathcal{P}_2$ to get the estimate under $(s, z)$ coordinates
\begin{equation*}
\|v-P_{x_0}\|_{L^\infty(\mathcal{P}_2(\Omega^1_{x_0}))}\le\bar{C}x_0^{-\qnt{m+2+\gamma}}\norm{\varphi}_{C^{(-m)}_{2, \gamma}([0, 1])}z(x_0, f(x_0))^{2+\gamma}.
\end{equation*}
Upon careful verification, we can find that the three coordinate transformations of Theorem \ref{thm:GlobalEstimates} are still valid under the assumptions \eqref{eq:fEstimatesW} and \eqref{eq:FiniteGrowConditionW}.
\end{proof}

As the proof of Theorem \ref{thm:LocalEstimate}, we have
\begin{thm}\label{thm:LocalEstimateW}
Under the assumption of Theorem \ref{thm:GlobalEstimatesW}, there exists $\sigma_0>0$ such that for any $\sigma\in (0,\sigma_0]$, we have 
\begin{equation*}
\norm{u}_{C^{(-m)}_{2, \gamma}(\overline \Omega\cap\set{x\in[0, 2/3]})}\le \bar{C}\norm{\varphi}_{C^{(-m)}_{2, \gamma}([0,1])},
\end{equation*}
where the two positive  constants $\sigma_0,\bar C$ depend only on $\gamma,\lambda,\Lambda, C_f, m$.
\end{thm}
\begin{proof}
Denote 
\begin{equation*}
w(x,y)=(v-P_{x_0})(x,y),\quad \widetilde w(\tilde x,\tilde y)=\frac{x_0^{m+2+\gamma}w(x,y)}{(f(x_0))^{2+\gamma}}
\end{equation*}
where the coordinate transformation is given by \eqref{eq:coordinatechange}.
Then 
\begin{equation*}
\begin{split}
-\widetilde {\mathcal L}\widetilde w=&\widetilde A\frac{x_0^{m+2+\gamma}(\varphi''(x_0+f(x_0)\tilde x)-\varphi''(x_0))}{(f(x_0))^\gamma}\\
+&\widetilde D\frac{x_0^{m+2+\gamma}(\varphi'(x_0+f(x_0)\tilde x)-\varphi'(x_0))}{(f(x_0))^\gamma}=\tilde g(\tilde x),
\end{split}\quad \text{in}\quad \widetilde \Omega_{x_0}^1
\end{equation*}
with the following boundary conditions
\begin{equation*}
\begin{split}
\|\widetilde w\|_{L^\infty(\widetilde \Omega_{x_0}^1)}+\|\widetilde w\|_{C^{2,\gamma}(\partial_+\widetilde \Omega_{x_0}^1)}\le \bar{C} \|\varphi\|_{C^{-m}_{2,\gamma}([0,1])},\quad \frac{\partial \widetilde w}{\partial \tilde y}|_{\partial_0\widetilde \Omega_{x_0}^1}=0.
\end{split}
\end{equation*}
Also we know 
\begin{equation*}
\|\tilde g\|_{C^\gamma(\overline{\widetilde \Omega_{x_0}^1})}\le \bar{C}\|\varphi\|_{C^{-m}_{2,\gamma}([0,1])}.
\end{equation*}
Similarly as \eqref{eq:HEstimateStandard}, one gets
\begin{equation}\label{eq:HEstimateStandardW}
\begin{split}
& \|\widetilde w\|_{C^{2,\gamma}(\overline{\widetilde \Omega_{x_0}^{1/2}})}\\
\le &\bar{C}\left(\|\tilde g\|_{C^\gamma(\overline{\widetilde \Omega_{x_0}^1})}+\|\widetilde w\|_{C^{2,\gamma}(\partial_+\widetilde \Omega_{x_0}^1)}+\|\widetilde w\|_{L^\infty(\widetilde \Omega_{x_0}^1)}\right)\le \bar{C}\|\varphi\|_{C^{-m}_{2,\gamma}([0,1])}. 
\end{split}
\end{equation}
In the original coordinate $(x,y)$, \eqref{eq:HEstimateStandardW} implies that $\forall x_0\in (0,2/3)$ there holds
\begin{equation*}
\begin{split}
&\max\limits_{i=0,1,2}\sup\limits_{(x, y)\in\Omega_{x_0}^{1/2}}\abs{x_0^{m+i}\nabla^{i} u(x, y)}\quad\\
+&\sup\limits_{(x_1, y_1),(x_2,y_2)\in\Omega_{x_0}^{1/2}}\abs{x_0^{m+k+\gamma}\frac{\nabla^{2} u(x_1, y_1)-\nabla^{2} u(x_2, y_2)}{\abs{(x_1, y_1)-(x_2, y_2)}^\gamma}}\le \bar{C} \|\varphi\|_{C^{-m}_{2,\gamma}([0,1])}.
\end{split}
\end{equation*}
In the following, we only need to consider 
\begin{equation*}
(x_1,y_1),(x_2,y_2)\in \overline{\Omega}\cap \{x\in (0,2/3)\},\quad x_2-x_1\ge \frac 12\max(f(x_1),f(x_2)).
\end{equation*}
Since $u_y(x,0)=0$, there holds 
\begin{equation*}
\begin{split}
&|u_{xy}(x_1,y_1)-u_{xy}(x_2,y_2)|\\
\le &|u_{xy}(x_1,0)-u_{xy}(x_1,y_1)|+|u_{xy}(x_2,0)-u_{xy}(x_2,y_2)|+|u_{xy}(x_1,0)-u_{xy}(x_2,0)|\\
\le & \bar{C}x_1^{-m-2-\gamma}\|\varphi\|_{C^{(-m)}_{2,\gamma}([0,1])}(y_1^\gamma+y_2^\gamma).
\end{split}
\end{equation*}
This implies 
\begin{equation*}
\|u_{xy}\|_{C^{(-m-2)}_{0,\gamma}(\bar \Omega\cap \{x\in [0,2/3]\})}\le \bar{C}\|\varphi\|_{C^{(-m)}_{2,\gamma}([0,1])}.
\end{equation*}

Notice that
\begin{equation*}
\begin{split}
u_{xx}+(f'(x))^2u_{yy}=&\frac{\dif^2 u(x,f(x))}{\dif x^2}-2f'(x)u_{xy}-f''(x)u_y\\
=&\frac{\dif^2 \varphi(x)}{\dif x^2}-2f'(x)u_{xy}-f''(x)u_y,
\end{split}
\end{equation*}
and \eqref{eq:PotentialEquations}
\begin{equation*}
Au_{xx}+Cu_{yy}=-Bu_{xy}-Du_x-Eu_y.
\end{equation*}
Hence for $\sigma_0$ small, there holds 
\begin{equation*}
\begin{split}
&|u_{xx}(x_1,f(x_1))-u_{xx}(x_2,f(x_2))|+|u_{yy}(x_1,f(x_1))-u_{yy}(x_2,f(x_2))|\\
\le & \bar{C}x_1^{-m-2-\gamma}\|\varphi\|_{C^{(-m)}_{2,\gamma}([0,1])}|(x_1,f(x_1))-(x_2,f(x_2))|^{\gamma}.
\end{split}
\end{equation*}
Then a similar argument as $u_{xy}$ yields that 
\begin{equation*}
\|u_{xx}\|_{C^{(-m-2)}_{0,\gamma}(\bar \Omega\cap \{x\in [0,2/3]\})}+\|u_{yy}\|_{C^{(-m-2)}_{0,\gamma}(\bar \Omega\cap \{x\in [0,2/3]\})}\le \bar{C}\|\varphi\|_{C^{(-m)}_{2,\gamma}([0,1])}.
\end{equation*}
 We complete our proof of Theorem \ref{thm:LocalEstimateW}.
\end{proof}

\begin{proof}[Proof of Theorem \ref{thm:GlobalEstimatesW}]
It suffices to show
\begin{equation*}
\norm{u}_{C^{2, \gamma}(\overline \Omega\cap\set{x\in[1/2, 1]})}\leq \bar{C}\norm{\varphi}_{C^{(-m)}_{2, \gamma}([0,1])}.
\end{equation*}
We cannot directly use the reflection transformation $\mathcal{P}_3$ to obtain the corresponding estimate as in the proof of Section 3. Because in the present section, the function blows up around the other vertex. The barrier function $H$ fails to work there. We need to make a modification to the proof.

In fact, by Theorem \ref{thm:LocalEstimateW}, we have
\begin{equation*}
\abs{u}\big|_{x=1/3}\leq \bar{C}\norm{\varphi}_{C^{(-m)}_{2, \gamma}([0,1])}.
\end{equation*}
Then for $x_0\in[1/2, 1]$
\begin{equation*}
\abs{v-P_{x_0}}\big|_{x=1/3}\leq \bar{C}\norm{\varphi}_{C^{(-m)}_{2, \gamma}([0,1])}.
\end{equation*}
And for the barrier function $H$ (the $r$ corresponds to a center at point $(x, y)=(1, 0)$), we have
\begin{equation*}
H|_{x=1/3}\geq \bar{C}>0.
\end{equation*}
Thus, we apply the maximum principle in the domain $\overline \Omega\cap\set{x\in[1/3, 1]}$ to obtain the estimate \eqref{eq:C2HEstimates}. The rest of the proof is the same.
\end{proof}

\section{Asymptotic States as $f\rightarrow0$}
In this section, we see that the estimate \eqref{eq:ObliqueDerivativeSchauder} leads to an asymptotic estimate as $f\rightarrow0$. First, in \eqref{eq:PotentialEquations}, we can differentiate the boundary value condition along the boundary to obtain the following derivatives list
\begin{equation*}
\begin{cases}
u(x, f(x))=\varphi(x),\\
l_fu(x, f(x))=u_x(x, f(x))+f'(x)u_y(x, f(x))=\varphi'(x),\\
l_0u(x, 0)=u_y(x, 0)+G(x)u_x(x, 0)=0,\\
L_fu(x, f(x))=u_{xx}(x, y)+2f'(x)u_{xy}(x, f(x))+\qnt{f'(x)}^2u_{yy}(x, f(x))\\
\quad\quad\quad\quad\quad\quad\quad+f''(x)u_y(x, f(x))=\varphi''(x),\\
Lu(x, y)=A(x, y)u_{xx}+B(x, y)u_{xy}+C(x, y)u_{yy}+D(x, y)u_{x}+E(x, y)u_{y}=0,\\
%=-A(x, 0)\varphi''-D(x, 0)\varphi'\\
L_0u(x, 0)=u_{xy}(x, 0)+G(x)u_{xx}(x, 0)+G'(x)u_x(x, 0)=0.
\end{cases}
\end{equation*}
In above, we set $f(x)=0$ to define the asymptotic operators $l_f^*, l_0^*, L_f^*, L^*$ and $L_0^*$ and the asymptotic state $u^*$ over $[0, 1]\times\set{0}$ as follows
\begin{equation*}
\begin{cases}
u^*(x, 0)=\varphi(x),\\
l_f^*u^*(x, 0)=u_x^*(x, 0)=\varphi'(x),\\
l_0^*u^*(x, 0)=l_0u^*(x, 0)=u_y^*(x, 0)+G(x)u_x^*(x, 0)=0,\\
L_f^*u^*(x, 0)=u_{xx}^*(x, 0)=\varphi''(x),\\
L^*u^*(x, 0)=A(x, 0)u_{xx}^*+B(x, 0)u_{xy}^*+C(x, 0)u_{yy}^*+D(x, 0)u_{x}^*+E(x, 0)u_{y}^*=0,\\
%=-A(x, 0)\varphi''-D(x, 0)\varphi'\\
L_0^*u^*(x, 0)=L_0u^*(x, 0)=u_{xy}^*(x, 0)+G(x)u_{xx}^*(x, 0)+G'(x)u_x^*(x, 0)=0.
\end{cases}
\end{equation*}
This is a linear equation of the following form derivatives $$\qnt{u^*(x, 0), u_x^*(x, 0), u_y^*(x, 0), u_{xx}^*(x, 0), u_{xy}^*(x, 0), u_{yy}^*(x, 0)}$$ and can be unique solved over $[0, 1]\times\set{0}$.

We can compare $u(x, y)$ and $u^*(x, 0)$ under these operators. For example,
\begin{equation*}
\begin{split}
\abs{L_f^*u(x, y)-L_f^*u^*(x, 0)}&=\abs{L_f^*u(x, y)-L_fu(x, f(x))}\\
&\leq\sigma_0\norm{u}_{C^{2, \gamma}(\Omega)}+f(x)^\gamma\norm{u}_{C^{2, \gamma}(\Omega)}\\
&\leq \sigma_0^\gamma\norm{\varphi}_{C^{2, \gamma}([0, 1])}.
\end{split}
\end{equation*}
Similarly, we have
\begin{equation*}
\begin{cases}
u(x, y)-u^*(x, 0)=\mathcal{O}_1,\\
l_f^*\qnt{u(x, y)-u^*(x, 0)}=\mathcal{O}_2,\\
l_0^*\qnt{u(x, y)-u^*(x, 0)}=\mathcal{O}_3,\\
L_f^*\qnt{u(x, y)-u^*(x, 0)}=\mathcal{O}_4,\\
L^*\qnt{u(x, y)-u^*(x, 0)}=\mathcal{O}_5,\\
L_0^*\qnt{u(x, y)-u^*(x, 0)}=\mathcal{O}_6,
\end{cases}
\end{equation*}
where $\abs{\mathcal{O}_i}\leq \bar{C}\sigma_0^\gamma\norm{\varphi}_{C^{2, \gamma}([0, 1])}$. Therefore, $\nabla^i u^*(x, 0)-\nabla^i u(x, y)$ satisfy a linear equation whose non-homogeneous terms can be estimated by $\bar{C}\sigma_0^\gamma\norm{\varphi}_{C^{2, \gamma}([0, 1])}$. To conclude, we have
\begin{thm}\label{thm:AsymptoticEstimate}
Under the assumption of Theorem \ref{thm:GlobalEstimates}, $u, u_x, u_y, u_{xx}, u_{xy}, u_{yy}$ is the perturbation of $u^*, u_x^*, u_y^*, u_{xx}^*, u_{xy}^*, u_{yy}^*$ and we have
\begin{equation}\label{eq:AsymptoticEstimate}
\abs{\nabla^i u^*(x, 0)-\nabla^i u(x, y)}\leq \bar{C}\sigma_0^\gamma\norm{\varphi}_{C^{2, \gamma}([0, 1])}.
\end{equation}
\end{thm}
\begin{rem}
Here there is an unreasonable definition. Because function $u^*$ is only defined on line segment $[0, 1]\times\set{0}$, how can we define its derivatives 
%$u^*_y, u^*_{xy}$ and $u^*_{yy}$ 
in the $y$ direction? In fact, it can be understood in two ways. On the one hand, we can consider some function $U$ defined in the neighbourhood of line segment $[0, 1]\times\set{0}$ and $\nabla^i U(x, 0)=\nabla^i u^*(x, 0).$ On the other hand, we can consider that there is a sequence of functions, each defined within the neighbourhood of line segment $[0, 1]\times\set{0}$, and the derivatives of the sequence of functions converge to the form derivatives of $u^*$ on $[0, 1]\times\set{0}$.
The introduction of $u^*$ has the benefit in providing a more optimal asymptotic estimate \eqref{eq:AsymptoticEstimate}  than the a priori estimate \eqref{eq:ObliqueDerivativeSchauder}.
\end{rem}

\section*{Acknowledgments}
The second author is partial supported by National Natural Science Foundation of China 12141105.

\end{document}